\documentclass[11pt,a4paper]{article}

\usepackage{amsmath,amsfonts,amssymb,latexsym,graphics,epsfig,url}
\usepackage{color}
\usepackage{amsthm}
\usepackage[english]{babel}
\usepackage{graphicx}
\usepackage{soul}
\usepackage[utf8]{inputenc}

\newcommand\blfootnote[1]{%
	\begingroup
	\renewcommand\thefootnote{}\footnote{#1}%
	\addtocounter{footnote}{-1}%
	\endgroup
}

\newtheorem{theorem}{Theorem}[section]
\newtheorem{proposition}[theorem]{Proposition}
\newtheorem{lemma}[theorem]{Lemma}
\newtheorem{corollary}[theorem]{Corollary}

\newcommand{\N}{\mathbb{N}}
\newcommand{\Z}{\mathbb{Z}}

\def\vv{\mbox{\boldmath $v$}}

\def\vec0{\mbox{\boldmath $0$}}

\def\A{\mbox{\boldmath $A$}}

\def\I{\mbox{\boldmath $I$}}

\def\U{\mbox{\boldmath $U$}}

\def\I{\mbox{\boldmath $I$}}

\def\N{\mbox{\boldmath $N$}}

\DeclareMathOperator{\rank}{rank}

\DeclareMathOperator{\spec}{sp}

\newcommand{\bigzero}{\mbox{\normalfont\Large\bfseries 0}}

\begin{document}

\title{Bipartite biregular Moore graphs\footnote{
The research of the first author is supported by CONACyT-M{\' e}xico under Project 282280 and PAPIIT-M{\' e}xico under Project IN101821, PASPA-DGAPA, and CONACyT Sabbatical Year 2020. The research of the second and third authors is partially supported by 
AGAUR from the Catalan Government under project 2017SGR1087 and by MICINN from the Spanish Government under project PGC2018-095471-B-I00. The research of the second and fourth authors is supported by MICINN from the Spanish Government under project MTM2017-83271-R.}}
\author{G. Araujo-Pardo$^a$, C. Dalf\'o$^b$, M. A. Fiol$^c$, N. L\'opez$^d$\\
\\
\small{$^a$Instituto de Matem\'aticas, Universidad Nacional Aut\'onoma de M\'exico}\\
{\small Mexico, \texttt{garaujo@math.unam.mx}}\\
\small{$^b$Departament de Matem\`{a}tica, Universitat de Lleida}\\
\small{Igualada (Barcelona), Catalonia, \texttt{cristina.dalfo@udl.cat}}\\
\small{$^c$Departament de Matem\`{a}tiques, Universitat Polit\`{e}cnica de Catalunya}\\
\small{Barcelona Graduate School of Mathematics}\\
\small{Institut de Matem\`atiques de la UPC-BarcelonaTech (IMTech)}\\
\small{Barcelona, Catalonia, \texttt{miguel.angel.fiol@upc.edu}}\\
\small{$^d$ Departament de Matem\`atica, Universitat de Lleida}\\
 \small{Lleida, Spain, \texttt{nacho.lopez@udl.cat}}}

\maketitle

\date{}


\blfootnote{
	\begin{minipage}[l]{0.3\textwidth} \includegraphics[trim=10cm 6cm 10cm 5cm,clip,scale=0.15]{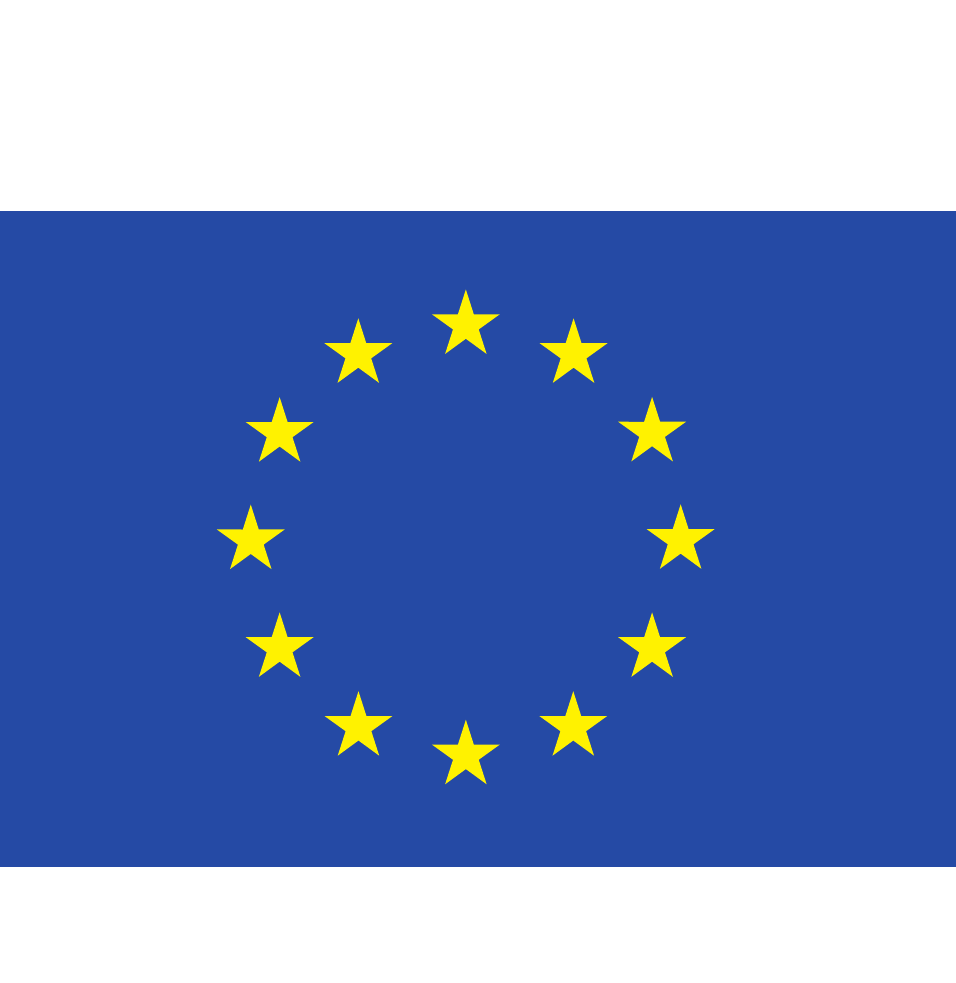} \end{minipage}  \hspace{-2cm} \begin{minipage}[l][1cm]{0.79\textwidth}
		The research of the second author has also received funding from the European Union's Horizon 2020 research and innovation programme under the Marie Sk\l{}odowska-Curie grant agreement No 734922.
\end{minipage}}

\noindent\emph{Keywords:} Bipartite biregular graphs, Moore bound, diameter, adjacency spectrum.\\
\emph{MSC2020:} 05C35, 05C50.


\begin{abstract}
A bipartite graph $G=(V,E)$ with $V=V_1\cup V_2$ is biregular if all the vertices of a stable set $V_i$ have the same degree $r_i$ for $i=1,2$. In this paper, we give an improved new Moore bound for an infinite family of such graphs with odd diameter. This problem was introduced in 1983 by Yebra, Fiol, and F\`abrega.\\
Besides, we propose some constructions of bipartite biregular graphs with diameter $d$ and large number of vertices $N(r_1,r_2;d)$, together with their spectra.
In some cases of diameters $d=3$, $4$, and $5$, the new graphs attaining the Moore bound are unique up to isomorphism.
\end{abstract}

\section{Introduction}
The {\emph{degree/diameter problem}} for graphs consists in finding the largest order of a graph with prescribed degree and diameter. We call this number the \emph{Moore bound}, and a graph whose order coincides with this bound is called a {\emph{Moore graph}}. 
	
	There is a lot of work related to this topic (see a survey by Miller and \v{S}ir\'a\v{n} \cite{MS16}), and also some restrictions of the original problem. One of them is related to the bipartite Moore graphs. In this case, the goal is finding regular bipartite graphs with maximum order and fixed diameter. In this paper, we study the problem, proposed by Yebra, Fiol and  F\`abrega \cite{YFF83} in 1983, that consists in finding biregular bipartite Moore graphs.  
	
	A bipartite graph $G=(V,E)$ with $V=V_1\cup V_2$ is biregular if, for $i=1,2$, all the vertices of a stable set $V_i$ have the same degree. We denote $[r,s;d]$-bigraph a bipartite biregular graph of degrees $r$ and $s$ and diameter $d$; and by  $[r,s;d]$-bimoore graph the bipartite biregular graph of diameter $d$ that attains the Moore bound, which is denoted    $M(r,s;d)$.
	Notice that constructing these graphs is equivalent to construct block designs, where one partite set corresponds to the points of the block design, and the other set corresponds to the blocks of the design. Moreover, each point is in a fixed number $s$ of blocks, and the size of each block is equal to $r$. The incidence graph of this block design is an $[r,s;d]$-biregular bipartite graph of diameter $d$.
	
	This type of graph is often used as an alternative to a hypergraph in modeling some interconnection networks. Actually, several applications deal with the study of bipartite graphs such that all vertices of every partite set have the same degree. For instance, in an interconnection network for a
	multiprocessor system, where the processing elements communicate through buses, it is useful that each processing element is connected to the same number of buses and also that each bus is connected to the same number of processing elements to have a uniform traffic through the network. These networks can be modeled by hypergraphs (see Bermond,  Bond, Paoli, and Peyrat \cite{BeBoPaPe83}), where the vertices indicate the processing elements and the edges indicate the buses of the system. They can also be modeled by bipartite graphs with a set of vertices for the processing elements, another one for the buses, and
	edges that represent the connections between processing elements and buses since all vertices of each set have the same degree.
	
	The degree/diameter problem is strongly related to the  {\em degree\/}/{\em girth problem\/} (also known as {\em the cage problem\/}) that consists in finding the smallest order of a graph with prescribed degree and girth (see the survey by Exoo and Jajcay \cite{ExooJaj08}). Note that when for an even girth of the graph, $g=2d$, the lower bound of this value coincides with the Moore bound for bipartite graphs (the largest order of a bipartite regular graph with given diameter $d$).
	
	In the bipartite biregular problem, we have the same situation. In 2019, Filipovski, Ramos-Rivera and Jajcay \cite{FilRamRivJaj19} introduced the concept of bipartite biregular Moore cages and presented lower bounds on the orders of bipartite biregular $(m,n;g)$-graphs. The bounds when $g=2d$ and $d$ even also coincide with the bounds given by Yebra, Fiol, and  F\`abrega in \cite{YFF83}. Note that these bounds only coincide when the diameter is even. The cases for odd diameter and girth $g=2d$ are totally different, even for the extreme values.

The contents of the  paper are as follows. In the rest of this introductory section, we recall the Moore-like bound  $M(r,s;d)$ derived in Yebra, Fiol, and F\`abrega \cite{YFF83} on the order of a bipartite biregular graph with degrees $r$ and $s$, and diameter $d$. Following the same problem of obtaining good bounds,  in Section \ref{sec:improvedMoore} we prove that, for some cases of odd diameter,  the Moore bound of \cite{YFF83} can be improved (see the new bounds in Tables \ref{tab:d=3} and \ref{tab:d=5}).
In the two following sections, we basically deal with the case of even diameter because known constructions provide optimal (or very good) results. Thus, Section \ref{gen-pols} is devoted to the Moore bipartite biregular graphs associated with generalized polygons.
In Section \ref{gen-cons}, we propose two general graph constructions: the subdivision graphs giving Moore bipartite biregular graphs with even diameter, and the semi-double graphs that, from a bipartite graph of any given diameter, allows us to obtain another bipartite graph with the same diameter but with a greater number of vertices. For these two constructions, we also give the spectrum of the obtained graphs.

	Finally, a numeric construction of bipartite biregular Moore graphs for diameter $d=3$ and degrees $r$ and $3$ is proposed in Section \ref{sec:d=3}. 

\subsection{Moore-like bounds}

Let $G=(V,E)$,  with $V=V_1\cup V_2$,  be a $[r,s;d]$-bigraph, where each vertex of $V_1$ has degree $r$, and each vertex of $V_2$ has degree $s$. Note that, counting in two ways the number of edges of $G$, we have
\begin{equation}
\label{basic=}
r N_1=s N_2,
\end{equation}
where $N_i=|V_i|$, for $i=1,2$.

Moreover, since $G$ is bipartite with diameter $d$, from one vertex $u$ in one stable set, we must reach {\bf all} the vertices of the other set in at most $k-1$ steps.
Suppose first that the diameter is even, say, $k=2m$ (for $m\ge 1$). Then, by simple counting, if $u\in V_1$, we get
\begin{equation}
\label{N2-even}
N_2\le r+r(r-1)(s-1)+\cdots+r[(r-1)(s-1)]^{m-1}= r\frac{[(r-1)(s-1)]^{m}-1}{(r-1)(s-1)-1},
\end{equation}
and, if $u\in V_2$,
\begin{equation}
\label{N1-even}
N_1\le s\frac{[(r-1)(s-1)]^{m}-1}{(r-1)(s-1)-1}.
\end{equation}
In the case of equalities in \eqref{N2-even}  and \eqref{N1-even}, condition \eqref{basic=} holds,
and the Moore bound is
\begin{equation}
\label{Moore-even}
M(r,s;2m)= (r+s)\frac{[(r-1)(s-1)]^{m}-1}{(r-1)(s-1)-1}.
\end{equation}
The Moore bounds for $2\le s\le r\le 10$ and $d=4,6$ are shown in Tables \ref{tab:d=4} and \ref{tab:d=6}, where values in boldface are known to be attainable.

\begin{table}[t]
\begin{center}
\begin{tabular}{|c||c|c|c|c|c|c|c|c|c|}
\hline
$r\setminus s$ & 2            & 3            & 4   & 5   & 6   & 7   & 8    & 9   & 10    \\
\hline \hline
2              & $\textbf{8}^{\bullet}$                                                      \\ \cline{1-3}
3              & $\textbf{15}^{* \diamond}$  & \textbf{30}                                               \\ \cline{1-4}
4              & $\textbf{24}^{* \diamond}$  & 49  & \textbf{80}                                         \\ \cline{1-5}
5              & $\textbf{35}^*$  & $\textbf{72}^{\bullet}$  & 117 & \textbf{170}                                  \\ \cline{1-6}
6              & $\textbf{48}^*$  & 99           & 160 & 231 & \textbf{312}                            \\ \cline{1-7}
7              & $\textbf{63}^*$  & 130          & 209 & 300 & 403 & $\stackrel{(518)}{516}$                      \\ \cline{1-8}
8              & $\textbf{80}^*$  & 165          & 264 & 377 & 504 & 645 & \textbf{800}                \\ \cline{1-9}
9              & $\textbf{99}^*$  & 204          & 325 & 462 & 615 & 784 & 969  & \textbf{1170}        \\ \cline{1-10}
10            & $\textbf{120}^*$ & 247          & $\textbf{292}^{\bullet}$ & 555 & 736 & 935 & 1152 & 1387 & \textbf{1640} \\ 
\hline
\end{tabular}
\end{center}
\vskip-.25cm
\caption{Moore bounds for diameter $d=4$. The attainable known values are in boldface.
The asterisks correspond to the subdivision graphs $S(K_{r,r})$, see Section \ref{sec:subdiv-graphs}.
The symbol `$\bullet$' indicates the orders of the graphs according to Theorem \ref{bb-Moore} and the diamonds correspond  to unique graphs.}
\vskip1cm
\label{tab:d=4}
\end{table}

\begin{table}[t]
\begin{center}
\begin{tabular}{|c||c|c|c|c|c|c|c|c|c|}
\hline
$r\setminus s$ & 2            & 3            & 4   & 5   & 6   & 7   & 8    & 9   & 10    \\
\hline \hline
2              & $\textbf{12}^{\bullet}$                                                      \\ \cline{1-3}
3              & \textbf{35*}  & \textbf{126}                                               \\ \cline{1-4}
4              & \textbf{78*} & 301   & \textbf{728}                                         \\ \cline{1-5}
5              &  \textbf{147*}  & 584    & 1431      & \textbf{2730}                                  \\ \cline{1-6}
6              &  \textbf{248*}  & 999    & 2410      & 4631       & \textbf{7812}                            \\ \cline{1-7}
7              &  387   & 1570   & 3773     & 7212       & 12103   &   $\stackrel{(18662)}{18660}$                    \\ \cline{1-8}
8              &  \textbf{570*}  & 2321   & 5556    & 10569      & 17654   & 27105      & \textbf{39216}                \\ \cline{1-9}
9              &  \textbf{803*}   & 3276   & $\textbf{7813}^{\bullet}$   &14798       & 24615   & 37648      & 54281     & \textbf{74898}        \\ \cline{1-10}
10             &  \textbf{1092*} & 4459  & 10598 & 19995      & 33136  &  50507      & 72594     & 99883          & \textbf{132860} \\ 
\hline
\end{tabular}
\end{center}
\vskip-.25cm
\caption{Moore bounds for diameter $d=6$. The attainable known values are in boldface. The asterisks correspond to the bimoore graphs of Proposition \ref{bb(s=2,d=6)-Moore}.  The `$\bullet$' indicates the order of the graph according to Theorem
\ref{bb-Moore}.}
\label{tab:d=6}
\end{table}

Similarly, if the diameter is odd, say, $k=2m+1$ (for $m\ge 1$), and $u\in V_1$, we have
\begin{equation}
\label{N1-odd}
N_1\le 1+r(s-1)+\cdots + r(s-1)[(r-1)(s-1)]^{m-1}=1+r(s-1)\frac{[(r-1)(s-1)]^{m}-1}{(r-1)(s-1)-1}=N_1',
\end{equation}
whereas, if $u\in V_2$,
\begin{equation}
\label{N2-odd}
N_2\le 1+s(r-1)\frac{[(r-1)(s-1)]^{m}-1}{(r-1)(s-1)-1}=N_2'.
\end{equation}
But, in this case, $N_1'r\neq N_2's$ and, hence, the Moore bound must be smaller than $N_1'+N_2'$. In fact, it was proved in Yebra, Fiol, and F\`abrega \cite{YFF83} that, assuming $r>s$,
\begin{equation}
\label{N1-N2-odd}
N_1\le \left\lfloor \frac{N_2'}{\rho}\right\rfloor \sigma\qquad\mbox{and}
\qquad N_2\le \left\lfloor \frac{N_2'}{\rho}\right\rfloor \rho,
\end{equation}
where $\rho=\frac{r}{\gcd\{r,s\}}$ and $\sigma=\frac{s}{\gcd\{r,s\}}$.
Then, in this case, we take the Moore bound
\begin{equation}
\label{Moore-odd}
M(r,s;2m+1)= \left\lfloor\frac{1+s(r-1)\frac{[(r-1)(s-1)]^{m}-1}{(r-1)(s-1)-1}}{\rho}\right\rfloor (\rho+\sigma).
\end{equation}


Two bipartite biregular graphs with diameter three attaining the Moore bound \eqref{Moore-odd}  were given in \cite{YFF83}. Namely, in Figure \ref{3unics}$(a)$, with $r=4$ and $s=3$, we would have the unattainable values $(N_1',N_2')=(9,10)$, whereas we get $(N_1,N_2)=(6,8)$, giving $M(4,3;3)=14$. In Figure \ref{3unics}$(b)$, with $r=5$ and $s=3$, $(N_1',N_2')=(11,13)$, and $(N_1,N_2)=(6,10)$, now corresponding to $M(5,3;3)=16$.

\begin{figure}[t]
	\centering
	\includegraphics[width=\linewidth]{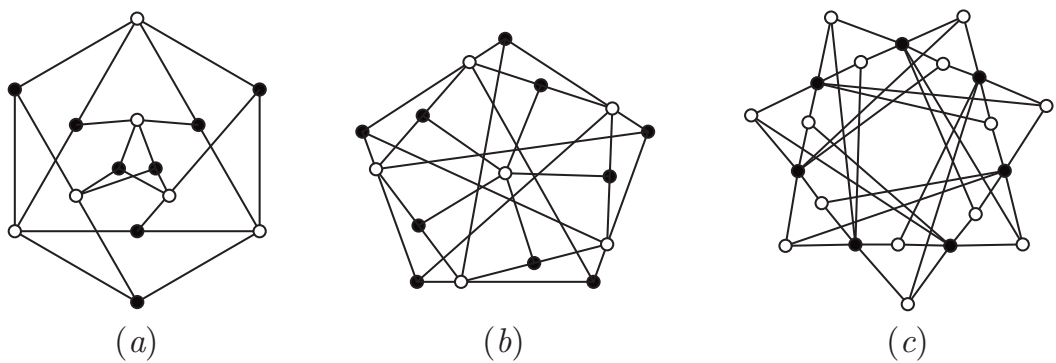}
	\vskip-6.75cm
	\caption{$(a)$ The only [4,3;3]-bimoore graph on $14$ vertices; $(b)$ One of the two [5,3;3]-bimoore graph on $16$ vertices; $(c)$ The only [6,3;3]-bimoore graph on $21$ vertices.}
	\label{3unics}
\end{figure}

As shown in Section \ref{sec:improvedMoore}, the bound \eqref{Moore-odd} can be improved  for some values of the degree $r$.
So, we display there the tables of the new Moore bounds for small values of $r,s$ and diameters $d=3$ and $d=5$.

Recall that
if $G=(V,E)$ is an $r$-regular graph of diameter $d$, then its \emph{defect} is $\delta=\delta(G)=M(r;d)-|V|$, where $M(r;d)$ stands for the corresponding Moore bound. Thus, in this paper, the defect of a $[r,s;d]$-bigraph $G=(V_1\cup V_2,E)$ is defined as $\delta=M(r,s;d)-|V_1\cup V_2|$.

\section{An improved Moore bounds for odd diameter}
\label{sec:improvedMoore}

Let us begin with a simple result concerning the girth of the  possible  biregular bipartite graphs attaining the bound \eqref{Moore-odd} for odd diameter.

\begin{lemma}
\label{g<=4m}
Every biregular bipartite graph $G$ of odd diameter $d=2m+1$ with order attaining the Moore bound \eqref{Moore-odd} has girth $g\le 4m$.
\end{lemma}
\begin{proof}
Since $G$ is bipartite, we must have $g\le 2d=4m+2$. Consider the trees $T_1$  rooted at $u\in V_1$ , and $T_2$ rooted at $v\in V_2$, of vertices at distance at most $2m$ of their roots. If  $G$ has girth $g=4m+2$, all vertices of $T_1$ must be different (otherwise, $T_1$ could not have the maximum number of vertices), and the same holds for all the vertices of $T_2$. This occurs if and only if $T_1$ and $T_2$ have numbers of vertices $N_1'$ and $N_2'$ given
by \eqref{N1-odd} and \eqref{N2-odd}, respectively.
But this is not possible because the Moore bound \eqref{Moore-odd} is obtained from \eqref{N1-N2-odd} as $\lfloor N_2'/\rho\rfloor(\rho+\sigma) <N_1'+N_2'$ (since $r>s\Rightarrow N_2'>N_1'$). Hence, the girth of $G$ is at most $4m$.
\end{proof}

\subsection{The case of diameter three}
As a consequence of the following result, we prove in Corollary \ref{coro:optimal} that the $[6,3;3]$-graph of Figure \ref{3unics}$(c)$ has the maximum possible order. 
\\
\begin{proposition}
\label{nonupperbound}
If $\rho=\frac{r}{\gcd\{r,s\}}$  divides $s-1$, then there is no $[r,s;3]$-graph with order attaining the Moore-like bound in \eqref{Moore-odd}.
Instead, the new improved Moore bound is
\begin{equation}
\label{Moore-odd-improved(d=3)}
M^*(r,s;3)= (1+s(r-1)-\rho)\left(1+\frac{\sigma}{\rho}\right),
\end{equation}
with
\begin{equation}
\label{N1-N2-(d=3)}
N_1\le [1+s(r-1)-\rho]\frac{\sigma}{\rho}
\qquad \mbox{and}\qquad N_2\le 1+s(r-1)-\rho,
\end{equation}
where $\rho=\frac{r}{\gcd\{r,s\}}$ and $\sigma=\frac{s}{\gcd\{r,s\}}$.
\end{proposition}

\begin{proof}
Suppose that, under the hypothesis, there exists a $[r,s;3]$-graph $G=(V_1\cup V_2,E)$ that attains the upper bound in  \eqref{Moore-odd}. Then, with $r>s$, $|V_1|$ is the number of vertices of degree $r$,  $ |V_2|$ is the number of vertices of degree $s$, and $N_1=|V_1|< |V_2|=N_2$.
Thus, for diameter $d=3$, we have
$$
N_2=\left\lfloor \frac{1+s(r-1)}{\rho}\right\rfloor\rho=1+s(r-1)=N_2'.
$$
This means that there is only one shortest path of length at most 2 from any vertex $v\in V_2$ to all the vertices of $V_2$.
Hence, the girth of $G$ is larger than 4. If not, we would have a cycle
$u_1\sim v_1\sim u_2\sim v_2 (\sim u_1)$, with $u_i\in V_1$ and $v_i\in V_2$ for $i=1,2$. So, there would be 2 shortest 2-paths between $v_1$ and $v_2$.
This is a contradiction with Lemma \ref{g<=4m} and, hence, the upper bound \eqref{Moore-odd} cannot be attained.

\end{proof}
Otherwise, if $\rho$ does not divide $s-1$, we have the bound \eqref{Moore-odd} with $m=1$, where \eqref{N1-N2-odd} yields
\begin{equation}
\label{N1-N2(d=3)not-divide}
N_1\le \left\lfloor s\cdot \gcd\{r,s\}-\frac{s-1}{\rho}\right\rfloor \sigma\qquad\mbox{and}\qquad N_2\le \left\lfloor s\cdot \gcd\{r,s\}-\frac{s-1}{\rho}\right\rfloor\rho.
\end{equation}

In Table \ref{tab:d=3}, we show the values of the Moore bounds in \eqref{Moore-odd} and \eqref{Moore-odd-improved(d=3)} for $s\le 2\le r\le 11$ and diameter $d=3$.
The values between parenthesis correspond to the old bound \eqref{Moore-odd} that was given in \cite{YFF83}.  The attainable known values are in boldface. The asterisks indicate the values obtained in this paper (see Proposition \ref{propo:(r,3)}). The $[6,3;3]$-bigraph with $21$ vertices of Figure \ref{3unics}$(c)$ can be shown to be unique, up to isomorphisms.

\begin{table}[t]
	\begin{center}
		\begin{tabular}{|c||c|c|c|c|c|c|c|c|c|c|}
			\hline
			$r\setminus s$ & 2            & 3   & 4   & 5   & 6   & 7   & 8    & 9   & 10  & 11  \\
			\hline \hline
			2   & \textbf{6}                                                      \\ \cline{1-3}
			3   & \textbf{5}   & \textbf{14}                                               \\ \cline{1-4}
			4   & \textbf{9}   & $\textbf{14}^{* \diamond}$       & \textbf{26}                                \\ \cline{1-5}
			5   & \textbf{7}   & \textbf{16*}            &  27  & \textbf{42}   \\ \cline{1-6}
			6   & \textbf{12} &  $\stackrel{(24)}{\textbf{21}^{\diamond}}$ & $\stackrel{(35)}{30}$                       &  44  & \textbf{62}                             \\ \cline{1-7}
			7   & \textbf{9}   & \textbf{20*}          & 33                       & 48  &  65 &  86                       \\ \cline{1-8}
			8   & \textbf{15} & \textbf{22*}          & 42                       & 52  & 70  &  90  & \textbf{114}                \\ \cline{1-9}
			9   & \textbf{11} & 32                           & 39                       & 56  & 80  & 96  &  119 & \textbf{146}          \\ \cline{1-10}
			10 & \textbf{18} & \textbf{26*}          &  49 & $\stackrel{(69)}{66}$  &  $\stackrel{(88)}{80}$  & 102 & 126 & 152 & \textbf{182}  \\ \cline{1-11}
			11 & \textbf{13} & \textbf{28*}          & 45                       & 64  & 85  & 108 & 133 & 160  & 189  & 222  \\
			\hline
		\end{tabular}
	\end{center}
	\vskip-.25cm
	\caption{Best Moore bounds for diameter $d=3$. The attainable known values are in boldface.
The asterisks correspond to the graphs obtained according to Proposition \ref{propo:(r,3)}, and the diamonds correspond to unique graphs. The values between parenthesis correspond to the old (unattainable) bound \eqref{Moore-odd}.}
\label{tab:d=3}
\vskip1cm
\end{table}

As a consequence of the above results, we get the following Moore bounds when the degree $r$ is a multiple of the degree $s$.
\begin{corollary}
For $s\ge 2$, the best Moore bounds for the orders $N_1$ and $N_2$ of a $[\rho s,s;3]$-graph are as follows:
\begin{itemize}
\item[$(i)$] If $\rho|(s-1)$, then
\begin{equation*}
N_1 \le (s^2-1)-\frac{s-1}{\rho},\qquad\mbox{and}\qquad N_2 \le \rho(s^2-1)-(s-1).
\end{equation*}
\item[$(i)$] If $\rho\nmid (s-1)$, then
\begin{equation*}
N_1 \le s^2-\left\lceil s/\rho\right\rceil,\qquad\mbox{and}\qquad
N_2 \le \rho(s^2-\left\lceil s/\rho\right\rceil).
\end{equation*}
\end{itemize}
\end{corollary}
\begin{proof}
Note that, under the hypothesis, $\gcd\{r,s\}=s$, $\rho=\frac{r}{s}$, and $s=1$.
Then, $(i)$ follows from \eqref{N1-N2-(d=3)}.
Concerning $(ii)$, the values in \eqref{N1-N2(d=3)not-divide} become
$$
N_1\le \left\lfloor s^2-\frac{s-1}{\rho}\right\rfloor\qquad\mbox{\rm and}\qquad
N_2\le \left\lfloor s^2-\frac{s-1}{\rho}\right\rfloor\rho,
$$
which are expressions equivalent to those given above.
\end{proof}

Assuming that $\rho=2$ and $s$ is odd, we get the following consequence.

\begin{corollary}
\label{coro:2s-s}
There is no $[2s,s;3]$-bimoore graph for $s$ odd.
\label{coro:optimal}
\end{corollary}
\begin{proof}
From the first statement, the Moore bound $M(6,3;3)=24$ given in \eqref{Moore-odd} is not attained. With this bound, we have $N_2'=16$ vertices of degree $3$, and $N_1'=8$ vertices of degree $6$. So, the first possible values are $N_2=14$ vertices of degree $3$ and $N_1=7$ vertices of degree $6$, which corresponds to the graph depicted in Figure \ref{3unics}$(c)$.
\end{proof}

\subsection{The general case}

Proposition \ref{nonupperbound} can be extended for any odd diameter, as shown in the following result.

\begin{theorem}\label{nonupperbound2m+1}
If $\rho=\frac{r}{\gcd\{r,s\}}$ divides $s\frac{[(r-1)(s-1)]^{m}-1}{(r-1)(s-1)-1}-1$, then there is no $[r,s;2m+1]$-graph with order attaining the Moore-like bound in \eqref{Moore-odd}. Instead, the new improved Moore bound is
\begin{equation}
\label{Moore-odd-improved}
M^*(r,s;2m+1)= \left(\frac{1+s(r-1)\frac{[(r-1)(s-1)]^{m}-1}{(r-1)(s-1)-1}}{\rho}-1\right) (\rho+\sigma),
\end{equation}
where, as before, $\rho=\frac{r}{\gcd\{r,s\}}$ and $\sigma=\frac{s}{\gcd\{r,s\}}$.
\end{theorem}

\begin{proof}
The proof that the bound in \eqref{Moore-odd} is not attainable follows the same reasoning as in Proposition \ref{nonupperbound}.
Indeed, from the hypothesis, if there exists a $[r,s;2m+1]$-graph $G=(V_1\cup V_2,E)$ with order attaining the upper bound in \eqref{Moore-odd}, we would have
$N_2=N_2'$, and the girth would be larger than $4m$, a contradiction.
Then, as before, the new possible bounds in \eqref{N1-N2-odd} are
$$
N_1\le \left( \frac{N_2'}{\rho}-1\right) \sigma\qquad\mbox{and}
\qquad N_2\le \left(\frac{N_2'}{\rho}-1\right) \rho,
$$
as claimed.	
\end{proof}

\begin{table}[t]
	\begin{center}
		\begin{tabular}{|c||c|c|c|c|c|c|c|c|c|}
			\hline
			$r\setminus s$ & 2   & 3   & 4    & 5    & 6     & 7    & 8    & 9    & 10     \\
			\hline \hline
			2              & 10                                                            \\ \cline{1-3}
			3              & $\textbf{15}^\diamond$  & 62                                                     \\ \cline{1-4}
			4              & 36  & $\stackrel{(112)}{105}$ & 242                                               \\ \cline{1-5}
			5              & 56  & 168 & 369  & 682                                        \\ \cline{1-6}
			6              & 80  & $\stackrel{(249)}{246}$ & $\stackrel{(535)}{530}$  & 957  & 1562                                \\ \cline{1-7}
			7              & 108 & 330 & 715 & $\stackrel{(1284)}{1272}$ & 2067  & 3110                        \\ \cline{1-8}
			8              & 140 & 429 & 924  & $\stackrel{(1651)}{1638}$ & 2646  & 3945 & 5602                \\ \cline{1-9}
			9              & 176 & 544 & $\stackrel{(1157)}{1144}$ & 2044  & 3280  & 4880 &  6885 & 9362          \\ \cline{1-10}
			10             & 216 & 663 & 1407 & $\stackrel{(2499)}{2496}$ & $\stackrel{(3976)}{3968}$  & 5882 & 8289 & 11229 & 14762 \\ 
			\hline
		\end{tabular}
	\end{center}
	\vskip-.25cm
	\caption{Best Moore bounds for diameter $d=5$. The diamond corresponds to unique graphs. The values between parenthesis correspond to old (unattainable)  bounds \eqref{Moore-odd}.
}
\label{tab:d=5}
\end{table}

Table \ref{tab:d=5} shows the values of the Moore bounds in \eqref{Moore-odd} and \eqref{Moore-odd-improved} for $s\le 2\le r\le 10$ and diameter $d=5$. As before, the values between parenthesis correspond to old (unattainable) bounds, as those in \eqref{Moore-odd}.

\subsection{Computational results}
The enumeration of bigraphs with maximum order can be done with a computer whenever the Moore bound is small enough. To this end, given a diameter $d$, first we generate with {\em Nauty} \cite{MP13} all bigraphs with maximum orders $N_1=|V_1|$ and $N_2=|V_2|$ allowed by the Moore bound $M(r,s;d)$. Second, we filter the generated graphs keeping those with diameter $d$ using the library {\em NetworkX} from Python. Computational resources forces to study the cases where $\max\{N_1,N_2\} \leq 24$ and $n=N_1+N_2 \leq 32$. Computational results are shown in Table \ref{tab:comput}. 
\begin{table}[t]
 \begin{center}
  \begin{tabular}{|c||c|c|c|c|c|}
   \hline
   $[r,s;d]$ & $n$ & $N_1$ & $N_2$ & Generated graphs & Graphs with diameter $d$ \\ \hline\hline
   $[4,3;3]$ & ${\bf 14}^\diamond$ & 6 & 8 & 18 & 1 \\ \hline
   $[5,3;3]$ & {\bf 16} & 6 & 10 & 45 & 2 \\ \hline
   
   $[6,3;3]$ & 24 & 8 & 16 & 977278 & 0 \\
             & ${\bf 21}^\diamond$ & 7 & 14 & 7063 & 1 \\ \hline
   $[7,3;3]$ & {\bf 20} & 6 & 14 & 344 & 4 \\ \hline
   $[8,3;3]$ & {\bf 22} & 6 & 16 & 950 & 10 \\ \hline
   $[9,3;3]$ & 32 & 8 & 24 & $>$122996904 & ? \\
             & 28 & 7 & 21 & 2262100 & 1 \\ \hline

   $[10,3;3]$ & {\bf 26} & 6 & 20 & 6197 & 19 \\ \hline
   $[11,3;3]$ & {\bf 28} & 6 & 22 & 14815 & 16 \\ \hline
   $[5,4;3]$ & 27& 12 & 15 & 822558 & 0 \\ 
   & {\bf 18}  & 8 & 10 & 3143 & 583 \\ \hline
    $[3,2;4]$ & ${\bf 15}^\diamond$ & 6 & 9 & 6 & 1 \\ \hline
    $[4,2;4]$ & ${\bf 24}^\diamond$ & 8 & 16 & 204 & 1 \\ \hline    
   $[3,2;5]$ & 20 & 8 & 12 & 20 & 0 \\ 
   & ${\bf 15}^\diamond$  & 6 & 9 & 6 & 1 \\ \hline
   
  \end{tabular}
 \end{center}
	\vskip-.25cm
	\caption{Complete enumeration of bipartite biregular Moore graphs for some (small) cases of the degrees $r,s$ and diameter $d$.}\label{tab:comput}
\end{table}

Even with the computational limitations on the order of the sets $V_1$ and $V_2$, some of the optimal values given in Tables \ref{tab:d=4}, \ref{tab:d=3}, and \ref{tab:d=5} have been found with this method. When there is no graph for the largest values $N_1$ and $N_2$, the following lower feasible pair values may decrease dramatically, producing a large number of optimal graphs, as it happens in the case $[5,4;3]$. The particular case $[9,3;3]$ is computationally very hard, and it is out of our computational resources (which are very limited), but our guess is that there is no Moore graph in this case. Finally, we point out that these results encourage us to study the case $s=d=3$ and $3 \nmid r$ more in detail (see Section \ref{sec:d=3}), where computational evidence shows that there is always a Moore graph.

\section{Bipartite biregular Moore graphs from generalized polygons}
\label{gen-pols}

In the first part of this section, we recall the connection between Moore graphs and generalized polygons that was extensively studied (see, for instance, Bamberg, Bishnoi, and Royle \cite{BamBisRoy}) because we will use it for the rest of the paper. In fact, our first result is an immediate consequence of the result proved by Araujo-Pardo, Jajcay, and Ramos in  \cite{AraJajRam}, and the analysis given in the introduction about the coincidence of the bounds for bipartite biregular cages and bipartite biregular Moore graphs when $d$ is even.
	
	\begin{theorem}\cite{AraJajRam}
		Whenever a generalized quadrangle, hexagon, or octagon 
		${\mathcal G}$ of order 
		$(s,t)$ exists, its point-line incidence graph is an $(s+1,t+1;8)$-,
		$(s+1,t+1;12)$- or $(s+1,t+1;16)$-cage, respectively.
		
		Hence, there exist infinite families of bipartite biregular 
		$(n+1,n^2+1;8)$-, $(n^2+1,n^3+1;8)$-, $(n,n+2;8)$-, $(n+1,n^3+1;12)$- and $(n+1,n^2+1;16)$-cages.
	\end{theorem}
	
	Then, immediately we conclude  the following result.
	\begin{theorem}
		\label{bb-Moore}
		There exists infinite families of bipartite biregular
		$[r^2+1,r+1;4]$-, $[r^3+1,r^2+1;4]$-, $[r+2,r;4]$-, $[r^3+1,r+1;6]$- and $[r^2+1,r+1;8]$-bimoore graphs.
	\end{theorem}
	
	In the following, we give some results related to generalized $n$-gons that we use in the rest of the paper. 
	
	\begin{lemma}[\cite{VM98}, Lemma 1.3.6]\label{VMLemma}
		A geometry $\Gamma = (\cal{P},\cal{L},{\bf I})$ is a (weak) generalized $n$-gon if and
		only if the incidence graph of $\Gamma$ is a connected bipartite graph of diameter $d=n$ and
		girth $g=2d$, such that each vertex is incident with at least three (at least two) edges.
	\end{lemma}
	As every generalized polygon $\Gamma$ can be associated with a pair $(r,s)$,
	called the {\em order} of $\Gamma$, such that every line is incident with $r+1$
	points, and every point is incident with $s+1$ lines (see van Maldeghem \cite{VM98}). This means,
	in particular, that the incidence graph of $\Gamma$ is a bipartite biregular graph
	with degrees $r+1$ and $s+1$. Besides, the following result
	determines the orders of both partite sets.
	
	\begin{theorem}[\cite{VM98}, Corollary 1.5.5]\label{PL}
		If there exists a $\Gamma = (\cal{P},\cal{L},{\bf I})$ (weak) generalized $n$-gon of order $(r,s)$ for $n \in \{3,4,6,8\}$, then
		
		\begin{itemize}
			\item
			For $n=3$,  $|{\cal{P}}|= r^2+r+1$\ and\ $|{\cal{L}}|=s^2+s+1$.
			\item
			For $n=4$,  $|{\cal{P}}|=(1+r)(1+rs)$\ and\ $|{\cal{L}}|=(1+s)(1+rs)$.
			\item
			For  $n=6$,  $|{\cal{P}}|=(1+r)(1+rs+r^2s^2)$\ and\  $|{\cal{L}}=(1+s)(1+rs+r^2s^2)$.
			\item
			For $n=8$,  $|{\cal{P}}|=(1+r)(1+rs)(1+r^2s^2)$\ and\ $|{\cal{L}}|=(1+s)(1+rs)(1+r^2s^2)$.
		\end{itemize}
		
	\end{theorem}
	
	In 1964, Feit and Higman proved that finite generalized $n$-gons exist only for $n=\{3,4,6,8\}$. When $n=3$, we have the projective planes; when $n=4$, we have the generalized quadrangles, which are known to exist for parameter pairs
	$(q,q), (q,q^2), (q^2,q),(q^2,q^3), $
	$(q^3,q^2), (q-1,q+1),(q+1,q-1) $; when $n=6$, we have the generalized hexagons with parameters
	$ (q,q),(q,q^3),(q^3,q)$, in both cases for $q$ prime power; and, finally, for $n=8$, we have the generalized octagons, which are only known to exist for the pairs $(q,q^2),(q^2,q)$, where $q$ is an odd power of $2$.

\section{Two  general constructions}
\label{gen-cons}

In this section, we construct some infinite families of Moore or large semiregular bipartite graphs derived from two general constructions: the subdivision graphs and the semi-double graphs

%

\subsection{The subdivision graphs}
\label{sec:subdiv-graphs}

Given a graph $G=(V,E)$, its {\em subdivision graph} $S(G)$ is obtained by inserting a new vertex in every edge of $G$. So, every edge $e=uv\in E$ becomes two new edges, $ux$ and $xv$, with new vertex $x$ of degree two ($\deg(x)=2$).

In our context, we have the following result.

\begin{proposition}
\label{propo:subdiv}
Let $G=(V_1\cup V_2,E)$ be an $r$-regular bipartite graph with $n$ vertices, $m$ edges,
diameter $D(\ge 2)$, and spectrum
$\spec G=\{\lambda_0^{m_0},\lambda_1^{m_1},\ldots,\lambda_{d-1}^{m_{d-1}},\lambda_d^{m_d}\}$,
where $\lambda_i=-\lambda_{d-i}$ and $m_i=m_{d-i}$ for $i=0,\ldots,\lfloor d/2\rfloor$. Then, its
{\em subdivision graph} $S(G)$, is a bipartite biregular graph with degrees $(r,2)$,  $n+m$ vertices, $2m$ edges, has  diameter $2D$, and spectrum
\begin{equation}
\label{spec-double}
\spec S(G)= \pm\sqrt{\spec G+r} \cup \{0^{m-n}\}.
\end{equation}
\end{proposition}
\begin{proof}
Let $S(G)$ have partite sets $U_1$ and $U_2$, with $U_1=V_1\cup V_2$ (the old vertices of $G$), and $U_2$ being the new vertices of degree 2. Let $x_{uv}\in U_2$ denote the vertex inserted in the previous edge $uv$.
The first statement is obvious. To prove that the diameter of $S(G)$ is $2m$, we consider three cases:
\begin{itemize}
\item[$(i)$]
Since $G$ has diameter $d$, there is a path of length $\ell \le d$ from any vertex $u\in U_1$ to any other vertex $v\in U_1$, say $u_0(=u),u_1,u_2,\ldots,u_{\ell}$. This path clearly induces a path of length $2\ell$ in $S(G)$. Namely, $u_0(=u),x_{u_0u_1},u_1,x_{u_1u_2}u_2,\ldots,x_{u_{\ell-1}u_{\ell}}u_{\ell}$.
\item[$(ii)$]
Since $G$ is bipartite, there is a path of length $\ell\le d-1$ between any two vertices in the same partite set (if $d$ is odd) or in different partite sets (if $d$ is even). Thus, from a vertex $x_{u_1u_2}\in U_2$,  there is a path of length $2\ell+1\le 2d-1$ to any vertex $v\in U_1$. (This is because either $u_1$ or $u_2$ are at distance $\ell\le d-1$.)
\item[$(iii)$]
Finally, a path of length at most $2d$ between vertices $x_{u_1u_2},x_{v_1v2}\in U_2$ is obtained by considering first the path from $x_{u_1u_2}$ to $v_i$, for some $i\in\{1,2\}$ (which, according to $(ii)$, has length at most $2d-1$), together with the edge $v_ix_{v_1v2}$.
\end{itemize}
The result about the spectrum of $S(G)$ follows from a result by Cvetkovi\'{c}~\cite{Cv75}, who proved that, if $G$ is an  $r$-regular graph with $n$ vertices and $m\big(=\frac{1}{2}nr\big)$ edges, then the characteristic polynomials of $S(G)$ and $G$ satisfy
$\phi_{S(G)}(x) = x^{m-n} \phi_{G}(x^2-r)$.
\end{proof}

For instance, the $[2,r;4]$-Moore graphs proposed in Yebra, Fiol, and F\`abrega \cite{YFF83}, with $N_1=2r$ and $N_2=r^2$ can be obtained as the subdividing graphs $S(K_{r,r})$ (see the values in column of $s=2$ in Table \ref{tab:d=4}).

For larger diameters, we can use the same construction with the known Moore bipartite graphs, which correspond to the incidence graphs of generalized polygons with $r=s$.

\begin{proposition}
\label{bb(s=2,d=6)-Moore}
For any value of $r\ge 3$, with $r-1$ a prime power, there exist three infinite families of bimoore graphs with corresponding parameters
$[r,2;2m]$ for $m\in\{3,4,6\}$.
\end{proposition}
\begin{proof}
According to \eqref{Moore-even}, a bipartite biregular Moore graph with degrees $r$ and $2$ and diameter $d=2m$ has order $M(r,2;d)=\frac{r+2}{r-2}[(r-1)^m -1]$. Then, from Proposition \ref{propo:subdiv}, these are the parameters obtained when considering the subdividing graph $S(G)$ of a bipartite Moore graph $G$ of degree $r$ and diameter $m$. (For diameter $d=6$, see the values in the column $s=2$ of Table \ref{tab:d=6}.)
\end{proof}

\subsection{The semi-double graphs}
Let $G$ be a bipartite graph with stable sets $V_1$ and $V_2$.
Given $i\in \{1,2\}$, the {\em semi-double(-$V_i$)  graph} $G^{2V_i}$ is obtained from $G$ by doubling each vertex of $V_i$, so that each vertex $u\in V_i$ gives rise to another vertex $u'$ with the same neighborhood as $u$, $G(u')=G(u)$. Thus, assuming, without loss of generality, that $i=1$, the graph $G^{2V_1}$ is bipartite with stable sets $V_1\cup V_1'$ and $V_2$, and satisfies the following result.
\begin{theorem}
\label{th:semi-double}
Let $G=(V_1\cup V_2,E)$ be a bipartite graph on $n=n_1+n_2=|V_1|+|V_2|$ vertices, diameter $D(\ge 2)$, and spectrum
$\spec G$.
Then, its semi-double graph $G^{2V_1}$, on $N=2n_1+n_2$ vertices, has the same diameter $D$, and spectrum
\begin{equation}
\label{spec-semi-double}
\spec G^{2V_1}=\sqrt{2}\cdot \spec G \cup \{0^{n_1}\}.
\end{equation}
\end{theorem}
\begin{proof}
Let $p : u_1(=u),u_2,\ldots,u_{\delta-1},u_{\delta}(=v)$ be a shortest path in $G$ between vertices $u$ and $v(\neq u')$, for $2\le \delta\le D$.
 If $u,v\in V_1\cup V_2$, $p$ also is a shortest path in $G^{2V_1}$. Otherwise, the following also are shortest paths in $G^{2V_1}$:
 \begin{align*}
  & u_1(=u),u_2,\ldots,u_{\delta-1},u_{\delta}'(=v');\\
  & u_1'(=u'),u_2,\ldots,u_{\delta-1},u_{\delta}(=v);\\
  & u_1'(=u'),u_2,\ldots,u_{\delta-1},u_{\delta}'(=v').
 \end{align*}
 Finally, the distance from $u$ to $u'$ is clearly two.
 \\
 To prove \eqref{spec-semi-double}, notice that the adjacency matrices $\A$ and $\A^{[1]}$ of $G$ and $G^{2V_1}$,  are
 $$
 \A=\left(
 \begin{array}{c|c}
 \bigzero & \N \\
 \hline \\
 [-.4 cm]
\N^{\top} & \bigzero
 \end{array}
 \right)\qquad \mbox{and}\qquad
 \A^{[1]}
 =\left(
 \begin{array}{c|c|c}
 \bigzero & \bigzero & \N \\
 \hline \\
 [-.4 cm]
 \bigzero & \bigzero & \N\\
 \hline \\
 [-.4 cm]
\N^{\top} & \N^{\top} & \bigzero
 \end{array}
 \right),
 $$
respectively, where $\N$ is an $n_1\times n_2$ matrix. Now, we claim that, if $\vv=(\vv_1 | \vv_2)^{\top}$ is a $\lambda$-eigenvector of $\A$, then $\vv^{[1]}=(\vv_1 | \vv_1| \sqrt{2}\vv _2)^{\top}$ is a $\sqrt{2}\lambda$-eigenvector of $\A^{[1]}$. Indeed, from
\begin{equation*}
 \A\vv=\left(
 \begin{array}{c|c}
 \bigzero & \N \\
 \hline \\
 [-.4 cm]
\N^{\top} & \bigzero
 \end{array}
 \right)
 \left(
 \begin{array}{c}
\vv_1\\
\hline
\vv_2
\end{array}
 \right)
 =\lambda \left(
 \begin{array}{c}
\vv_1\\
\hline
\vv_2
\end{array}
 \right)
 \qquad \Rightarrow \qquad \N\vv _2=\lambda \vv _1\quad{\rm and}\quad \N^{\top}\vv _1=\lambda \vv _2,
 \end{equation*}
we have
$$
 \A^{[1]}\vv^{[1]}=
 \left(
 \begin{array}{c|c|c}
 \bigzero & \bigzero & \N \\
 \hline \\
 [-.4 cm]
 \bigzero & \bigzero & \N \\
 \hline \\
 [-.4 cm]
\N^{\top} & \N^{\top} & \bigzero
 \end{array}
 \right)
\left(
\begin{array}{c}
\vv_1\\
\hline
\vv_1\\
\hline \\
 [-.45 cm]
\sqrt{2}\vv _2
\end{array}
 \right)=
 \left(
\begin{array}{c}
\sqrt{2}\N\vv _2\\
\hline
\sqrt{2}\N\vv _2\\
\hline \\
 [-.45 cm]
2\N^{\top}\vv _1
\end{array}
 \right)
\left(
\begin{array}{c}
\sqrt{2}\lambda \vv _1\\
\hline
\sqrt{2}\lambda \vv _1\\
\hline
2\lambda \vv _2
\end{array}
\right)
=\sqrt{2}\lambda \vv ^{[1]}.
$$
Moreover, if $\U=\left(
\begin{array}{c}
\U_1\\
\hline
\U_2
\end{array}
\right)$
is a matrix whose columns are the $n$ independent eigenvectors of $\A$, then the $n$ columns of the matrix
 $\U'=\left(
\begin{array}{c}
\U_1\\
\hline
\U_1\\
\hline \\
 [-.45 cm]
\sqrt{2}\U_2
\end{array}
\right)$
also are independent since, clearly, $\rank \U'=\rank \U=n_1+n_2$.
Consequently, $\spec \A\subset \spec \A^{[1]}$. Finally, each of the remaining $n_1$ eigenvalues $0$ corresponds to an eigenvector with $u$-th component $+1$ and $u'$-th component $-1$ for a given $u\in V_1$, and $0$ elsewhere. This is because the matrix $\U^{[1]}$ obtained by extending $\U'$  with such eigenvalues, that is,
 $$
 \U^{[1]}=\left(
\begin{array}{c|c}
\U_1 & \I \\
\hline
\U_1 & -\I \\
\hline \\
 [-.45 cm]
\sqrt{2}\U_2 & \bigzero
\end{array}
\right)
$$
has rank $n=2n_1+n_2$, as required.
\end{proof}

In general, we can consider the $k$-tuple graph $G^{kV_i}$, which is defined as expected by replacing each vertex $u\in V_i$ of $G$  by $k$ vertices $u_1,\ldots,u_k$ with the same adjacencies as $u$. Then, similar reasoning as in the proof of Theorem 
\ref{th:semi-double} leads to the following result.
\\
\begin{theorem}
\label{th:k-tuple}
Let $G=(V_1\cup V_2,E)$ be a bipartite graph on $n=n_1+n_2=|V_1|+|V_2|$ vertices, diameter $D(\ge 2)$, and spectrum
$\spec G$.
Then, its $k$-tuple graph $G^{kV_1}$, on $N=kn_1+n_2$ vertices, has the same diameter $D$, and spectrum
\begin{equation}
\label{spe-k-yuble}
\spec G^{kV_1}=\sqrt{k}\cdot \spec G \cup \{0^{(k-1)n_1}\}.
\end{equation}
\end{theorem}

As a consequence of Theorem \ref{th:semi-double}, we introduce a family of $[r,2r;d]$-graphs for $d=\{3,4,6\}$ using the existence of bipartite Moore graphs of order $M(r;d)$ for $r-1$ a prime power and these values of $d$. That is, the incidence graphs of the mentioned generalized polygons.

\begin{theorem}
\label{r2r}
The following are $[r,2r;d]$-biregular bipartite graphs for $r\geq 3$, $r-1$ a prime power, and diameter $d\in \{3,4,6\}$.
\begin{itemize}
\item [(i)]
A $[r,2r;3]$-biregular bipartite graph has order $n=3r^2-3r+3$ with defect $\delta=\frac{3}{2}(r-1)$ for odd $r$, and $\delta=\frac{3r}{2}-3$ for even $r$.
\item [$(ii)$]
A $[r,2r;4]$-biregular bipartite graph has order $n=3r^3-6r^2+10r$ with defect  $\delta=3r^3-3r^2-2r$.
\item [$(iii)$]
A $[r,2r;6]$-biregular bipartite graph has order $n=2r^5-8r^4+14r^3-12r^2+6r$ with defect  $\delta=10r^5-28r^4+31r^3-15r^2-3r$.
\end{itemize}
\end{theorem}

\begin{proof}
$(i)$ Let $G$ be a $[r,3]$-Moore graph, that is, the incidence graph of a projective plane of order $r-1$. As already mentioned, $G$ is bipartite, has $2r^2-2r+1$ vertices, and diameter $3$.
Then, by Theorem \ref{th:semi-double}, the semi-double graph $G^{[1]}$ (or $G^{[2]}$) is
a bipartite graph on $3r^2-3r+3$ vertices, biregular with degrees $r$ and $2r$, and diameter $3$.
Thinking on the projective plane, this corresponds to duplicate, for instance, each line, so that each point is on $2r$ lines, and each line has $r$ points, as before. Then,  $G^{[1]}$ is the incidence graph of this new incidence geometry ${\cal{I}}_G$.
%
Moreover, the Moore bound \eqref{Moore-even} is
$M(2r,r;3)=3r^2-\frac{3}{2}(r-1)$ if $r$ is odd, and $M(2r,r;3)=3r^2-\frac{3}{2}r$ if $r$ is even. Thus,  a simple calculation gives that the defect of $G$ is equal to $\delta=\frac{3}{2}(r-1)$ for $r$ odd and $\delta=\frac{3}{2}r-3$ for $r$ even.

Figure \ref{3unics}$(c)$ depicts the only $[6,3;3]$-biregular bipartite graph of order $21$ obtained from the Heawood graph (that is, the incidence graph of the Fano plane).
Notice that, by Corollary \ref{coro:2s-s}, this is a Moore graph since it has the maximum possible number of vertices (see Table \ref{tab:d=3}).


$(ii)$ In this case, we apply Theorem \ref{th:semi-double} to  $G$ being the $[r;4]$-Moore graph (that is, the incidence graph of a generalized quadrangle of order $r-1$). Now the semi-double graph $G^{[1]}$ is $[r,2r;4]$-bipartite biregular 
with $3[(r-1)^3+(r-1)^2+(r-1)+1]=3r^3-6r^2+10r$ vertices, whereas the Moore bound in  \eqref{Moore-even} is  $M(2r,r;3)=6r^3-9r^2+6r$. Hence, the defect is $\delta=3r^3-3r^2-2r$.
For example, the $[3,6;4]$-bigraph of order $45$ obtained from the Tutte graph (that is, the incidence graph of the generalized quadrangle of order $2$) has defect $\delta=54$.


$(iii)$ Finally, to obtain a  $[r,2r;6]$-bigraph of order $2r^5-8r^4+14r^3-12r^2+6r$, we apply Theorem \ref{th:semi-double} to the $[r;6]$-Moore graph (the incidence graph of a generalized hexagon of order $r-1$). Then, we obtain a $[r,2r;6]$-bigraph with
$2r^5-8r^4+14r^3-12r^2+6r$
vertices, whereas  the Moore bound in \eqref{Moore-even} is $M(2r,r;6)=12r^5-36r^4+45r^3-27r^2+9r$, yielding a defect $\delta=10r^5-28r^4+31r^3-15r^2-3r$.
\end{proof}

Notice that,  in Theorem \ref{r2r}, we only state results for Moore graphs constructed with regular generalized quadrangles; clearly, it is possible to do the same starting with biregular bipartite Moore graphs of diameters $6$ and $8$ given as in Theorem \ref{bb-Moore}, but since the bounds are far to be tight, we
restricted the details to the best cases.

\section{Bipartite biregular Moore graphs of diameter $3$}
\label{sec:d=3}

We have already seen some examples of large biregular graphs with diameter three. Namely, the Moore graphs with degrees $(3,4)$, $(3,5)$, and $(3,6)$ of Figure  \ref{3unics} (the last one in Theorem \ref{r2r}$(i)$).
Now, we begin this section with the simple case of bimoore graphs with degrees  $r$ (even)  and $2$, and diameter $3$. For these values, the Moore bounds in \eqref{Moore-odd} and \eqref{Moore-even} turn out to be $M(r,2;3)=2+r$ when $r(>1)$ is odd, and $M(r,2;3)=3\left(1+\frac{r}{2}\right)$ ($N_1=3$ and $N_2=3r/2$) when $r(>2)$ is even, respectively.
In the first case, the bound is attained by the complete bipartite graph $K_{2,r}$ and, hence, the diameter is, in fact, $2$.
In the second case, the Moore graphs are obtained via Theorem \ref{th:k-tuple}. Let $G$ be a hexagon (or $6$-cycle) with vertex set $V_1\cup V_2=\{u_1,u_3,u_5\}\cup \{2,4,6\} $. Then,  the $k$-tuple $G^{kV_1}$ with $k=r/2$ is a $[2,2r;3]$-bimoore graph on $3(r+1)$ vertices (see Table \ref{tab:d=3}).

In general, and in terms of designs, to guarantee that the diameter is equal to $3$, it is necessary that any pair of points shares a block, and any pair of blocks has a non-empty intersection.
It is well known that this kind of structure exists when $r=m$ and $r-1$ is a  prime power. Namely, the so-called projective plane of order $r-1$. In this case, any pair of blocks (or lines) intersects in one and only one point and, for any pair of points, they share one and only one block (or line).
 For more details about projective planes, you can consult, for instance, Coxeter \cite{Cox93}.  In our constructions of block designs giving graphs of diameter $3$, this condition is not necessary, and two blocks can be intersected in more than one point, and two points can share one or more blocks. Moreover, we may have the same  block appearing more than once.

\subsection{Existence of bipartite biregular Moore graphs of diameter 3 and degrees $r\ge 3$ and  $s=3$}
\label{sec:d=3}

Let us consider the diameter $3$ case together with $r>s=3$. According to Eqs. \eqref{N1-odd} and \eqref{N2-odd}, the maximum numbers of vertices for the partite sets, in this case, are $N_1'=2r+1$ and $N_2'=3r-2$. We recall that the equality $N_1'r=3N_2'$ does not hold when the diameter is odd. Since
$$
\rho=\frac{r}{\gcd\{r,3\}}=\left\{\begin{array}{cl}
                              r & \textrm{if} \ \ 3 \nmid r \\
                              \frac{r}{3} & \textrm{otherwise}
                             \end{array}\right.
\quad \textrm{and} \quad \sigma=\frac{3}{\gcd\{r,3\}}=\left\{\begin{array}{cl}
                              3 & \textrm{if} \ \ 3 \nmid r \\
                              1 & \textrm{otherwise}
                             \end{array}\right.
$$
we obtain the Moore-like bounds (see \eqref{N1-N2-odd}):
\[
N_1 \leq \left\lfloor \frac{3r-2}{\rho}\right\rfloor \sigma = \left\{\begin{array}{cl}
                              3\omega & \textrm{if} \ \ 3 \nmid r \\
                              \omega' & \textrm{otherwise}
                             \end{array}\right.
\quad \textrm{and} \quad N_2\le \left\lfloor \frac{3r-2}{\rho}\right\rfloor \rho = \left\{\begin{array}{cl}
                              r\omega & \textrm{if} \ \ 3 \nmid r \\
                              \frac{r\omega'}{3} & \textrm{otherwise}
                             \end{array}\right.
\]
where $\omega=\left\lfloor \frac{3r-2}{r}\right\rfloor=2$ and $\omega'=\left\lfloor \frac{3r-2}{r/3}\right\rfloor=8$. As a consequence, whenever $3 \nmid r$, we obtain Moore-like bounds $N_1\leq 6$ and $N_2 \leq 2r$. Otherwise, when $3|r$, we have
the bounds $N_1\leq 8$ and $N_2 \leq \frac{8}{3}r$.

The following graphs have orders that either attain or are close to such Moore bounds.
Given an integer $n\ge 6$, let $G_{6+n}=(V_1 \cup V_2, E)$ be the bipartite graph with independent sets $V_1=\{(0,j) \ | \ j \in \mathbb{Z}_6\}$, $V_2=\{(1,i) \ | \ i \in \mathbb{Z}_{n}\}$, and where $(1,i) \sim (0,j)$ for all $i \in \mathbb{Z}_{n}$ if and only if $j \equiv i \pmod{6}$ or $j \equiv i+1 \pmod{6}$ or $j \equiv i+3 \pmod{6}$.

Thus, $G_{6+n}$ is a bipartite graph with $n+6$ vertices, where every vertex $v\in V_2$ has degree $3$ and,  assuming that $n=6k+\rho$ ($n\equiv \rho \pmod{6}${\color{blue})}, vertex $u\in V_1$ has the degree indicated in Table  \ref{tab:degreesV1}.
Indeed, let us consider the case $\rho=2$ (the other cases are analogous), where $n=6k+2$ and, for simplicity, let  $V_2= \{0,1,\ldots,6k-1\}$. According to the adjacency rules,
the vertex $(0,i)\in V_1$, for $i\in \Z_6$, is adjacent to all the vertices $ j\in V_2$ with $j\equiv i,i-1,i-3 \pmod{6}$. Then, we have the following cases:
\begin{itemize}
\item
 If $i=0$, $V_2$ contains $k+1$ numbers $j\equiv 0 \pmod6$, $k$ numbers $j\equiv-1\equiv 5 \pmod6$, and $k$ numbers $j\equiv-3\equiv 3 \pmod6$. Thus, the degree of $(0,0)$ is $3k+1$.
\item
 If $i=1$, $V_2$ contains $k+1$ numbers $j\equiv 1 \pmod6$, $k+1$ numbers $j\equiv0 \pmod6$, and $k$ numbers $j\equiv-2\equiv 4 \pmod6$. Thus, the degree of $(0,2)$ is $3k+2$.
\item
 If $i=2$, $V_2$ contains $k$ numbers $j\equiv 2 \pmod6$, $k+1$ numbers $j\equiv1 \pmod6$, and $k$ numbers $j\equiv-1\equiv 5 \pmod6$. Thus, the degree of $(0,3)$ is $3k+1$.
\item[ \vv dots]
\item
 If $i=5$, $V_2$ contains $k$ numbers $j\equiv 5 \pmod6$, $k$ numbers $j\equiv4 \pmod6$, and $k$ numbers $j\equiv2 \pmod6$. Thus, the degree of $(0,5)$ is $3k$.
\end{itemize}
\begin{table}[h]
	\begin{center}
		\begin{tabular}{|c||c|c|c|c|c|c|}
			\hline
			$\rho\setminus u$ &  (0,0)   & (0,1) & (0,2)   & (0,3)    &  (0,4)     & (0,5)      \\
			\hline
			0              &  $3k$ & $3k$ & $3k$ & $3k$ & $3k$ & $3k$ \\
			1              & $3k+1$  & $3k+1$ & $3k$ & $3k+1$ &  $3k$ & $3k$ \\
			2              & {\boldmath  $3k+1$}  &  {\boldmath $3k+2$} &  {\boldmath $3k+1$} &  {\boldmath $3k+1$} &  {\boldmath $3k+1$} &  {\boldmath $3k$}   \\
			3               & $3k+1$  & $3k+2$ & $3k+2$ & $3k+2$ & $3k+1$ & $3k+1$   \\
			4              & {\boldmath $3k+2$}  & {\boldmath $3k+2$} & {\boldmath $3k+2$} & {\boldmath $3k+3$} & {\boldmath $3k+2$} & {\boldmath $3k+1$}   \\
			5              & $3k+2$  & $3k+3$ & $3k+2$ & $3k+3$ & $3k+3$ & $3k+2$   \\
			\hline
		\end{tabular}
	\end{center}
	\caption{Degrees of the vertices $u=(0,j)\in V_1$ when $|V_2|=n=6k+\rho$. }
	\label{tab:degreesV1}
\end{table}

\begin{proposition}
\label{G(6+n)}
The diameter of the bipartite graph $G_{6+n}=(V_1\cup V_2, E)$, on $n+6$ vertices, is $d=3$.
\end{proposition}
\begin{proof}
It suffices to prove that every pair of different vertices $u,u'\in V_1$, and every pair of different vertices $v,v'\in V_2$, are at distance two.
In the first case, notice that $u=(0,j)\in V_1$ is adjacent to every vertex $v=(1,i)\in V_2$ such that $i\equiv j,j-1, j-3 \pmod6$.
Moreover, vertex $(1,j)\in V_2$ is adjacent to vertices $(0,j+1),(0,j+3)\in V_1$, vertex $(1,j-1)\in V_2$ is adjacent to vertices
 $(0,j-1),(0,j+2)\in V_1$, and  vertex $(1,j-3)\in V_2$ is adjacent to vertices
 $(0,j-3),(0,j-2)\in V_1$. Schematically (with all arithmetic modulo 6),
\begin{align}
(0,j)\quad & \sim \quad (1, j), (1,j-1), (1, j-3) \label{u-u'}\\
 & \sim \quad (0,j+1),(0,j+3),(0,j-1),(0,j+2),(0,j-3),(0,j-2) \nonumber,
\end{align}
which are all vertices of $V_1$ different from $(0,j)$.
Similarly, starting from a vertex  of $V_2$, we have
 \begin{align*}
(1,i)\quad & \sim \quad (0, i), (0,i+1), (0, i+3)\\
 & \sim \quad (1,i), (1,i-1),(1,i-3),(1,i+1),(1,i-2),(1,i+3),(1,i+1),
\end{align*}
with the last two representing all vertices of $V_2$ because the second entries cover all values modulo 6. This completes the proof.
\end{proof}

The following result proves the existence of an infinite family of Moore bipartite biregular graphs with diameter three.
\begin{proposition}
\label{propo:(r,3)}
For any integer $r\ge 6$  such that $3 \nmid r$, there exists a bipartite graph $G$ on $2r+6$ vertices (the Moore bound for the case of degrees $(r,3)$ and diameter $3$) with degrees $3,r,r\pm 1$ and diameter $d=3$.
Moreover, when $r\equiv 2 \mod3$, there exists a Moore bipartite biregular graph with degrees $(r,3)$  and diameter $3$.
\end{proposition}
\begin{proof}
From the comments at the beginning of this subsection, if $3 \nmid r$, the graph  $G_{6+n}$ of Proposition \ref{G(6+n)} with $n=2r$  has maximum order for diameter $3$. However, as shown before, not every vertex of $u\in V_1$ has degree $r$ since it is required to become a bipartite biregular Moore graph. More precisely,  if $2r=6k+\rho$ (with $\rho=2$ or $\rho=4$, since $3 \nmid r$), from the two rows in boldface of Table \ref{tab:degreesV1}, we have
 $$
\deg(u)=\left\{\begin{array}{cl}
                             r & \textrm{if} \ \ u \in \{(0,0),(0,2),(0,3),(0,4)\}, \\
                              r+1 & \textrm{if} \ \ u=(0,\rho-1), \\
                              r-1 & \textrm{if} \ \ u=(0,5). \\
                             \end{array}\right.
$$
This proves the first statement.

When  $r\equiv 2 \pmod3$, that is $\rho=4$, we obtain biregularity by modifying only one adjacency of $G_{6+n}$, as follows. Let $G'_r$ be the graph $G_{6+r}$ defined above, but where the edge $(0,3)\sim(1,i)$, for some $i\equiv 0 \pmod 3$, is switched to $(0,5)\sim(1,i)$. Then, $G'_r$ is a bipartite semiregular Moore graph of diameter $3$ for all $r\ge 5$.
To prove that $G'_r$ has diameter $d=3$, let us check again that every pair of different vertices $u,u'\in V_1$, and every pair of different vertices $v,v'\in V_2$, are at distance two.
\begin{itemize}
\item
If  $u,u'\in V_1$, we only need to consider the case when $u=(0,3)$. Then, the paths are as follows (where $(1,i)$ represents any vertex $(1,i')$ with $i'\equiv i \pmod6$):
\begin{align*}
(0,3)\quad & \sim \quad (1, 3), (1,2), (1, 0)\quad \sim \quad (0,4),(0,0),(0,2),(0,5),(0,1),
\end{align*}
because $(0,3)$ was initially adjacent to more than one vertex of type $(0,i)$ with  $i\equiv 0 \pmod 3$. Thus, all vertices of $V_0$ are reached from $(0,3)$.
\item
If  $v,v'\in V_2$, assume that $v=(1,i)$ with $i\equiv 0 \pmod6$ (the case where $i\equiv 3 \pmod6$ is similar). Then,
\begin{itemize}
\item
If $v=(1,0)$, we have the paths
\begin{align*}
(1,0)\quad & \sim \quad (0, 0), (0,1), (0,5)
 \quad \sim \quad (1,0),(1,3),(1,1),(1,4),(1,2).
\end{align*}
 Notice that, in this case, the first step is not $(1,0)\sim (0,3)$ (deleted edge), but $(1,0)\sim (0,5)$. Despite this, we still reach all vertices of $V_2$, as required.
 \item
 If $v=(1,i)$ with $i\neq 0$, the first step is
 \begin{align*}
(1,i)\quad & \sim \quad (0, i), (0,i+1), (0,i+3)
\end{align*}
So, the only problem would be when some $i,i+1,i+3$ is $3$, since we have not the adjacency $(0,3)\sim (1,0)$. But, if so, we have the following alternative adjacencies: if $i=0,3$, we  have $(0,0)\sim (1,0)$; and if $i=2$, we  have $(0,5)\sim (1,0)$.
Again, all vertices of $V_2$ are reached, completing the proof.
\end{itemize}
\end{itemize}
\vskip-.5cm
 \end{proof}

%
%
%


\end{document}